\documentclass[11pt]{amsart}
\usepackage{amsmath,amssymb,amsthm,enumerate,graphicx,mathtools}

\usepackage{xypic} 
\usepackage{mathrsfs} 
\usepackage{xspace}

\usepackage{hyperref}

\usepackage[usenames, dvipsnames]{color}
\newtheorem{thm}{Theorem}[section] 

\newtheorem{obs}[thm]{Observation}
\newtheorem{prop}[thm]{Proposition}
\newtheorem{lem}[thm]{Lemma}
\newtheorem{cor}[thm]{Corollary}

\theoremstyle{definition}
\newtheorem{defn}[thm]{Definition}

\newtheorem{rmk}[thm]{Remark}

\newtheorem*{claim*}{Claim}
\newtheorem{ack}[thm]{Acknowledgments}

\newcommand{\mc}[1]{\mathcal{#1}}
\newcommand{\mb}[1]{\mathbb{#1}}
\newcommand{\mf}[1]{\mathfrak{#1}}

\newcommand{\Zb}{\mathbb{Z}}

\newcommand{\Rb}{\mathbb{R}}

\newcommand{\Tb}{\mathbb{T}}

\newcommand{\Norm}{\mathcal{N}}

\newcommand{\tdlc}{t.d.l.c.\@\xspace}

\newcommand{\lcsc}{l.c.s.c.\@\xspace}
\newcommand{\defbold}{\textbf}

\newcommand{\normal}{\trianglelefteq}

\newcommand{\inv}{^{-1}}
\newcommand{\triv}{\{1\}}

\newcommand{\CC}{\mathrm{C}}
\newcommand{\N}{\mathrm{N}}
\newcommand{\Z}{\mathrm{Z}}

\newcommand{\cgrp}[1]{\overline{\langle #1 \rangle}}

\newcommand{\grp}[1]{\langle #1 \rangle}
\newcommand{\ol}[1]{\overline{#1}}

\newcommand{\Sym}{\mathop{\rm Sym}\nolimits}
\newcommand{\Aut}{\mathop{\rm Aut}\nolimits}

\newcommand{\Rad}[1]{\mathop{\rm Rad}_{#1}\nolimits}

\begin{document}

\title[The essentially chief series]{The essentially chief series of a compactly generated locally compact group}
\author{Colin D. Reid}
\address{  University of Newcastle,
   School of Mathematical and Physical Sciences,
   University Drive,
   Callaghan NSW 2308, Australia}
\email{colin@reidit.net}
\thanks{The first named author is an ARC DECRA fellow.  Research supported in part by ARC Discovery Project DP120100996.}

\author{Phillip R. Wesolek}
\address{Binghamton University, Department of Mathematical Sciences, PO Box 6000, 	Binghamton, New York 13902-6000, USA}

\address{Previous address: Universit\'{e} catholique de Louvain,
   Institut de Recherche en Math\'{e}matiques et Physique (IRMP),
   Chemin du Cyclotron 2, bte L7.01.02,
   1348 Louvain-la-Neuve, Belgium}
\thanks{The second named author was supported by ERC grant \#278469.}
\email{pwesolek@binghamton.edu}

\begin{abstract}
We first obtain finiteness properties for the collection of closed normal subgroups of a compactly generated locally compact group. Via these properties, every compactly generated locally compact group admits an essentially chief series -- i.e.\ a finite normal series in which each factor is either compact, discrete, or a topological chief factor. A Jordan--H\"{o}lder theorem additionally holds for the `large' factors in an essentially chief series.
\end{abstract}
\maketitle
\tableofcontents

\addtocontents{toc}{\protect\setcounter{tocdepth}{1}}
\section{Introduction}

Within the theory of locally compact groups, an important place is occupied by the compactly generated groups.  From a topological group theory perspective, every locally compact group is the directed union of its compactly generated subgroups, so problems of a `local' nature can be reduced to the compactly generated case. From a geometric perspective, such groups admit a well-defined geometric structure and are the natural generalization of finitely generated groups. Finally and most concretely, many examples of locally compact groups of independent interest are compactly generated. Any locally compact group that acts continuously, properly, and cocompactly on a proper metric space is compactly generated; for example $\Aut(\Gamma)$ is such a group for \textit{any} Cayley graph $\Gamma$ of a finitely generated group.

There is an emerging structure theory of compactly generated locally compact groups which reveals that they have special properties, often in a form that has no or trivial counterpart in the theory of finitely generated discrete groups. This theory could be said to begin with the paper \textit{Decomposing locally compact groups into simple pieces} of P.-E. Caprace and N. Monod, \cite{CM11}, in which general results on the normal subgroup structure of compactly generated locally compact groups are derived. 

The key insight of Caprace and Monod is to study compactly generated locally compact groups as \textit{large-scale topological objects}. That is to say, they observe that non-trivial interactions between local structure and large-scale structure place significant restrictions on compactly generated locally compact groups. Of course, these restrictions will always be up to compact groups and discrete groups; e.g.\ these results are insensitive to, say, taking a direct product with a discrete group. We stress that this theory nonetheless can yield non-trivial results for discrete groups; consider \cite{SW_13} or \cite{W17}.

The work at hand is a further contribution to the (large-scale topological) structure theory of compactly generated locally compact groups. We first establish finiteness conditions for the lattice of closed normal subgroups. These conditions are then used to prove the existence of a finite chief series, up to compact groups and discrete groups, in any compactly generated locally compact group. 

\begin{rmk}Compactly generated locally compact groups are second countable modulo a compact normal subgroup; see \cite[Theorem 8.7]{HR79}. We thus restrict to the second countable case whenever convenient.
\end{rmk}

\subsection{Statement of results} 
A \defbold{normal factor} of a topological group $G$ is a quotient $K/L$ such that $K$ and $L$ are distinct closed normal subgroups of $G$ with $L < K$. We say that $K/L$ is a (topological) \defbold{chief factor} of $G$ if there are no closed normal subgroups of $G$ strictly between $L$ and $K$.

\begin{defn}
An \defbold{essentially chief series} for a locally compact group $G$ is a finite series
\[
\triv = G_0 \leq G_1 \leq \dots \leq G_n = G
\]
of closed normal subgroups such that each normal factor $G_{i+1}/G_i$ is either compact, discrete, or a topological chief factor of $G$.
\end{defn}

Every compactly generated locally compact group $G$ admits an essentially chief series; indeed, any finite normal series can be refined to an essentially chief series.

\begin{thm}[See Theorem~\ref{thm:chief_series}]
Suppose that $G$ is a compactly generated locally compact group. If $(G_1,G_2,\dots,G_m)$ is an increasing sequence of closed normal subgroups of $G$, then there exists an essentially chief series
\[
\triv = K_0 \leq K_1 \leq \dots \leq K_n = G
\]
for $G$ such that $\{G_1,\dots,G_m\}$ is a subset of $\{K_0,\dots,K_n\}$.
\end{thm}

\begin{cor}Every compactly generated locally compact group admits an essentially chief series.
\end{cor}

A Jordan--H\"{o}lder theorem additionally holds for essentially chief series. For our Jordan--H\"{o}lder theorem, the \textit{association} classes of chief factors are uniquely determined:

\begin{defn}[See \cite{RW_P_15}]
For a topological group $G$, normal factors $K_1/L_1$ and $K_2/L_2$ are \defbold{associated} if $\ol{K_1L_2} = \ol{K_2L_1}$ and $K_i \cap \ol{L_1L_2} = L_i$ for $i=1,2$.
\end{defn}
\noindent If chief factors $K_1/L_1$ and $K_2/L_2$ of $G$ are associated, then there is a third normal factor $K/L$ of $G$ such that each $K_i/L_i$ admits a $G$-equivariant continuous monomorphism into $K/L$ with a dense normal image; see \cite[Lemma 6.6]{RW_P_15}. Associated chief factors are not in general isomorphic as topological groups. The association relation is also not an equivalence relation in general, but it becomes one when restricted to non-abelian chief factors; see \cite[Proposition 6.8]{RW_P_15}.  

We must also ignore certain small chief factors.
\begin{defn}
For an \lcsc group $G$, a chief factor $K/L$ is called \textbf{negligible} if it is either abelian or associated to a compact or discrete factor. 
\end{defn}

\begin{thm}[see Theorem~\ref{thm:association_map}]
Suppose that $G$ is a locally compact second countable group and that $G$ has two essentially chief series $(A_i)_{i=0}^m$ and $(B_j)_{j=0}^n$. Define
\begin{align*}
I &:= \{ i \in \{1,\dots,m\} \mid A_i/A_{i-1} \text{ is a non-negligible chief factor of $G$}\}; and \\
J &:= \{ j \in \{1,\dots,n\} \mid B_j/B_{j-1} \text{ is a non-negligible chief factor of $G$}\}.
\end{align*}
Then there is a bijection $f:I \rightarrow J$ where $f(i)$ is the unique element $j \in J$ such that $A_i/A_{i-1}$ is associated to $B_j/B_{j-1}$. 
\end{thm}

In \cite{RW_LC_16}, we show negligible non-abelian chief factors are limited in topological complexity. Specifically, they are either connected and compact or totally disconnected with dense quasi-center. We also show the chief factors themselves have a tractable internal normal subgroup structure.

\begin{rmk}
The essentially chief series seems to be a useful tool with which to study compactly generated locally compact groups. Via the existence of the series, questions can often be reduced to chief factors, which are topologically characteristically simple. The essentially chief series can also be used to establish non-existence results. For example, the uniqueness given by Theorem~\ref{thm:association_map} ensures any finitely generated just infinite branch group has no infinite commensurated subgroups of infinite index; see \cite{W17}. 
\end{rmk}

Our results follow from a finiteness property of the lattice of closed normal subgroups, which generalizes results of Caprace--Monod in \cite{CM11}. The essential tool is given by the \defbold{Cayley--Abels graph}; this graph is a connected locally finite graph on which the group in question acts vertex-transitively with compact open point stabilizers. The \emph{finite degree} of such graphs along with the usual dimension from Lie theory provide the finiteness from which all other finiteness properties in this paper are deduced.

A family of closed normal subgroups $\mc{F}$ is \textbf{filtering} if for all $N,M\in \mc{F}$, there is $K\in \mc{F}$ with $K\leq N\cap M$. A family $\mc{D}$ is \textbf{directed} if for all $N,M\in \mc{D}$, there is $K\in \mc{D}$ with $N\cup M\subseteq K$.

\begin{thm}[See Theorem~\ref{thm:essential_finiteness}]\label{thmintro:essential_finiteness} Let $G$ be a compactly generated locally compact group.
\begin{enumerate}[(1)]
\item If $\mc{F}$ is a filtering family of closed normal subgroups of $G$, then there exists $N \in \mc{F}$ and a closed normal subgroup $K$ of $G$ such that $\bigcap \mc{F} \le K \le N$, $K/\bigcap \mc{F}$ is compact, and $N/K$ is discrete.

\item If $\mc{D}$ is a directed family of closed normal subgroups of $G$, then there exists $N \in \mc{D}$ and a closed normal subgroup $K$ of $G$ such that $N \le K \le \cgrp{\mc{D}}$, $K/N$ is compact, and $\cgrp{\mc{D}}/K$ is discrete.
\end{enumerate}
\end{thm}

Theorem~\ref{thmintro:essential_finiteness} implies additionally the existence of interesting quotients. For a property of groups $P$, a topological group $G$ is called \textbf{just-non-$P$} if $G$ does not have $P$, but every proper non-trivial Hausdorff quotient $G/N$ has property $P$.

\begin{thm}[See Theorem~\ref{thm:just-non-P}]
Let $P$ be a property of compactly generated locally compact groups.  If all groups with $P$ are compactly presented and $P$ is closed under compact extensions, then for any compactly generated locally compact group $G$, exactly one of the following holds:
	\begin{enumerate}[(1)]
		\item Every non-trivial quotient of $G$ (including $G$ itself) has $P$; or 
		\item $G$ admits a quotient that is just-non-$P$.
	\end{enumerate}
\end{thm}
Any quasi-isometry invariant property is stable under compact extensions, so these properties occasionally fall under the previous theorem. 

\begin{cor}
Let $G$ be a compactly generated locally compact group that does not have polynomial growth. Then $G$ admits a non-trivial quotient that is just-not-(of polynomial growth).
\end{cor}

\begin{ack}
Many of the ideas for this project were developed during a stay at the Mathematisches Forschungsinstitut Oberwolfach; we thank the institute for its hospitality. We also thank the anonymous referees for their detailed suggestions.
\end{ack}

\section{Preliminaries}
\subsection{Notations and generalities}

All groups are taken to be Hausdorff topological groups and are written multiplicatively. Topological group isomorphism is denoted by $\simeq$. We use ``t.d.", ``l.c.", and ``s.c." for ``totally disconnected", ``locally compact", and ``second countable."

The collection of closed normal subgroups of a topological group $G$ is denoted by $\Norm(G)$.  The connected component of the identity is denoted by $G^\circ$. For any subset $K\subseteq G$, $\CC_G(K)$ is the collection of elements of $G$ that centralize every element of $K$. We denote the collection of elements of $G$ that normalize $K$ by $\N_G(K)$. The topological closure of $K$ in $G$ is denoted by $\ol{K}$. If $G$ acts on a set $X$, $G_{(x)}$ denotes the stabilizer of $x\in X$ in $G$.

A topological group is \textbf{Polish} if it is separable and admits a complete, compatible metric. A locally compact group is Polish if and only if it is second countable; cf. \cite[(5.3)]{K95}.

For a poset $\mc{P}$, a \defbold{filtering family} $\mc{F}\subseteq \mc{P}$ in $\mc{P}$ is a subset of $\mc{P}$ such that for all $N,M\in \mc{F}$, there exists $L\in \mc{F}$ with $L \le M$ and $L \le N$. Dual to this notion, $\mc{D}\subseteq \mc{P}$ is a \defbold{directed family} if for all $M,N\in \mc{D}$, there is $L\in \mc{D}$ with $M \le L$ and $N \le L$.

\subsection{Chief factors and chief blocks}\label{sec:chief_factor}
We here recall the basic theory established in \cite{RW_P_15}. In the present work, this theory is lightly used to establish the Jordan--H\"{o}lder theorem.

\begin{defn}
A \defbold{normal factor} of a topological group $G$ is a quotient $K/L$ such that $K$ and $L$ are distinct closed normal subgroups of $G$ with $L < K$. We say that $K/L$ is a (topological) \defbold{chief factor} of $G$ if there are no closed normal subgroups of $G$ strictly between $L$ and $K$. 
\end{defn}

There is a natural notion of `equivalence' between chief factors:
\begin{defn}
For a topological group $G$, normal factors $K_1/L_1$ and $K_2/L_2$ are \defbold{associated} if $\ol{K_1L_2} = \ol{K_2L_1}$ and $K_i \cap \ol{L_1L_2} = L_i$ for $i=1,2$.
\end{defn}

Association is not an equivalence relation in general, but it becomes one when restricted to the set of non-abelian chief factors of a topological group $G$; see \cite[Proposition 6.8]{RW_P_15}.  For a non-abelian chief factor $K/L$, the equivalence class of non-abelian chief factors equivalent to $K/L$ is denoted by $[K/L]$. The class $[K/L]$ is called a \textbf{chief block} of $G$, and the set of chief blocks of $G$ is denoted by $\mf{B}_G$.

Given a Polish group $G$ and normal subgroups $N \le M$ of $G$, we say that $M/N$ \defbold{covers} $[K/L]$ if there exist closed normal subgroups $N \le B < A \le M$ of $G$ for which $A/B$ is a non-abelian chief factor associated to $K/L$.  Otherwise we say that $M/N$ \defbold{avoids} $[K/L]$.  We say that $M$ covers or avoids $[K/L]$ if $M/\triv$ does.

Our Jordan--H\"{o}lder theorem is a consequence of the following general refinement theorem.

\begin{thm}[{\cite[Theorem 1.14]{RW_P_15}}]\label{thmintro:Schreier_refinement}
Let $G$ be a Polish group, $K/L$ be a non-abelian chief factor of $G$, and
\[
\triv = G_0 \le G_1 \le \dots \le G_n = G
\]
be a series of closed normal subgroups in $G$.  Then there is exactly one $i \in \{0,\dots,n-1\}$ such that $G_{i+1}/G_i$ covers $[K/L]$.\end{thm}

\subsection{Background on locally compact groups} A closed subgroup $K$ of a locally compact group $G$ is \defbold{cocompact} if the coset space $G/K$ is compact when equipped with the quotient topology. 

A locally compact group $G$ is \defbold{locally elliptic} if every finite subset of $G$ is contained in a compact subgroup.  The  \defbold{locally elliptic radical}, denoted by $\Rad{\mc{LE}}(G)$, is the union of all closed normal locally elliptic subgroups of $G$. 

\begin{thm}[Platonov, \cite{Plat66}]\label{thm:platonov_radical}
For $G$ a locally compact group, $\Rad{\mc{LE}}(G)$ is the unique largest locally elliptic closed normal subgroup of $G$, and 
\[
\Rad{\mc{LE}}(G/\Rad{\mc{LE}}(G)) = \triv.
\] 
\end{thm}

A (real) \defbold{Lie group} is a topological group that is a finite-dimensional analytic manifold over $\Rb$ such that the group operations are analytic maps. A Lie group $G$ can have any number of connected components, but $G^\circ$ is always an \emph{open} subgroup of $G$. The group $G/G^\circ$ of components is thus discrete.

\begin{thm}[Gleason--Yamabe; see {\cite[Theorem 4.6]{MZ}}]\label{thm:yamabe_radical}
Let $G$ be a locally compact group. If $G/G^\circ$ is compact, then $\Rad{\mc{LE}}(G)$ is compact, and the quotient $G/\Rad{\mc{LE}}(G)$ is a Lie group with finitely many connected components.
\end{thm}

Theorem~\ref{thm:yamabe_radical} suggests a notion of dimension applicable to all locally compact groups. 
\begin{defn}
For a locally compact group $G$, the \defbold{non-compact real dimension}, denoted by $\dim^\infty_{\Rb}(G)$, is the dimension of $G^\circ/\Rad{\mc{LE}}(G^\circ)$ as a real manifold.
\end{defn}
The non-compact real dimension is always finite. It is \textit{superadditive}, not subadditive, with respect to extensions. Additionally, $\dim^\infty_{\Rb}(G) = 0$ if and only if $G$ is compact-by-(totally disconnected).

The following technical consequence of Theorem~\ref{thm:yamabe_radical} will be useful later:

\begin{lem}\label{lem:sgrp_con}
Suppose $G$ is a locally compact group with closed normal subgroups $H<L$. If $L$ is connected, then there is a closed normal subgroup $K\normal G$ such that $ H^{\circ}\leq K\leq H$ with $K/H^{\circ}$ compact and $H/K$ discrete. In particular, $K/K^{\circ}$ is compact.
\end{lem}
\begin{proof}
Let $R:=\Rad{\mc{LE}}(L)$. The group $R$ is a compact normal subgroup of $G$ and is such that $L/R$ is a Lie group via Theorem~\ref{thm:yamabe_radical}. The group $HR/R$ is a closed subgroup of the Lie group $L/R$, hence $HR/R$ is a Lie group. Additionally, the connected component of $HR/R$ equals $H^{\circ}R/R$. The group $K:=H^{\circ}R\cap H$ satisfies the lemma.
\end{proof}

The basic structural properties of Lie groups stated in the proposition below are classical and will be used without further comment.

\begin{prop}\label{prop:basic_Lie_facts}\
\begin{enumerate}[(1)]
\item The only connected abelian Lie groups are groups of the form $\Rb^a \times \Tb^b$ for $a,b\geq 0$ where $\Tb := \Rb/\Zb$ is the circle group. 

\item The factors in the closed derived series of a connected solvable Lie group are themselves connected abelian Lie groups.

\item A Lie group $L$ has a largest connected solvable normal subgroup, called the \defbold{solvable radical} of $L$.  

\item A connected Lie group with trivial solvable radical is \defbold{semisimple}.  A semisimple Lie group $L$ has discrete center, and $L/\Z(L)$ is a finite direct product of abstractly simple groups.

\item Every closed subgroup of a Lie group is a Lie group.

\item The dimension of a Lie group $L$ regarded as a real manifold, denoted by $\dim_{\mb{R}}(L)$, is additive with respect to extensions: if $\dim_{\mb{R}}(L)=n$ and $K$ is a closed normal subgroup of dimension $k$, then $\dim_{\mb{R}}(L/K)=n-k$.
\end{enumerate}
\end{prop}

We shall need one further observation about abelian Lie groups.

\begin{lem}\label{connected_abelian_approximation}
Let $A$ be a connected abelian Lie group.  If $A = \cgrp{\mc{D}}$ where $\mc{D}$ is a directed family of closed subgroups of $A$,  then some $D \in \mc{D}$ is cocompact in $A$.
\end{lem}

\begin{proof}
Write $A$ as $A = \Rb^a \times \Tb^b$ for some non-negative integers $a$ and $b$.  Since $\Tb^b$ is compact, we can pass to the quotient $A/\Tb^b$ and assume that $A = \Rb^a$.  The conclusion now follows by considering the $\Rb$-linear span of $D$ for each $D \in \mc{D}$.
\end{proof}

\subsection{Cayley--Abels graphs}
Cayley--Abels graphs play an essential role in the present work. Our discussion of Cayley--Abels graphs is somewhat more technical than usual. This additional complication is necessary to ensure the degree behaves well under quotients. 

A \defbold{graph} $\Gamma = (V,E,o,r)$ consists of a vertex set $V = V\Gamma$, a directed edge set $E = E\Gamma$, a map $o:E \rightarrow V$ assigning to each edge an \textbf{initial vertex}, and a bijection $r: E \rightarrow E$, denoted by $e \mapsto \ol{e}$ and called \textbf{edge reversal}, such that $r^2 = \mathrm{id}$. 

The \textbf{terminal vertex} of an edge is $t(e) := o(\ol{e})$. A \defbold{loop} is an edge $e$ such that $o(e) = t(e)$.  For $e$  a loop, we allow both $\ol{e} = e$ and $\ol{e} \not= e$ as possibilities.  The \defbold{degree} of a vertex $v \in V$ is $\deg(v):=|o\inv (v)|$, and the graph is \defbold{locally finite} if every vertex has finite degree.  The \defbold{degree} of the graph is defined to be 
\[
\deg(\Gamma) := \sup_{v \in V\Gamma} \deg(v).
\]
The graph is \defbold{simple} if the map $E \rightarrow V \times V$ defined by $ e \mapsto (o(e),t(e))$ is injective and no edge is a loop. 

An \defbold{automorphism} of a graph is a pair of permutations $\alpha_V: V \rightarrow V$ and $\alpha_E: E \rightarrow E$ that respect initial vertices and edge reversal: $\alpha_V(o(e))=o(\alpha_E(e))$ and $\ol{\alpha_E(e)}=\alpha_E(\ol{e})$. For simple graphs, automorphisms are just permutations of $V$ that respect the edge relation in $V\times V$.

For $G$ a group acting on a graph $\Gamma$ and $v\in V\Gamma$, the orbit of $v$ under $G$ is denoted by $Gv$. We denote the orbit of an edge $e\in E\Gamma$ under $G$ by $Ge$. The action of $G$ gives a \defbold{quotient graph} $\Gamma/G$ as follows: the vertex set $V_G$ is the set of $G$-orbits on $V$ and the edge set $E_G$ is the set of $G$-orbits on $E$. The origin map $\tilde{o}:E_G\rightarrow E_G$ is defined by $\tilde{o}(Ge):=Go(e)$; this is well-defined since graph automorphisms send initial vertices to initial vertices. The reversal $\tilde{r}:E_G\rightarrow E_G$ is given by $Ge\mapsto G\ol{e}$; this map is also well-defined. We will abuse notation and write $o$ and $r$ for $\tilde{o}$ and $\tilde{r}$. 

We stress an important feature of quotient graphs: If $N$ is a normal subgroup of $G$, then $\Gamma/N$ is naturally equipped with an action of $G$ with kernel containing $N$. The action of $G$ on $\Gamma/N$ therefore factors through $G/N$.

\begin{lem}\label{lem:quotient:basic}Let $G$ be a group acting on a graph $\Gamma$ with $N$ a normal subgroup of $G$.
\begin{enumerate}[(1)]
\item If $\deg(\Gamma)$ is finite, then $\deg(\Gamma/N) \le \deg(\Gamma)$, with equality if and only if there exists a vertex $v \in V$ of maximal degree such that the elements of $o\inv (v)$ all lie in distinct $N$-orbits.
\item For $v \in V$, the vertex stabilizer in $G$ of $Nv$ is $NG_{(v)}$.
\end{enumerate}
\end{lem}

\begin{proof}
$(1)$ Take $v \in V\Gamma$ and let $Ne$ be an edge of $\Gamma/N$ such that $o(Ne) = Nv$.  There then exists $v' \in Nv$ and $e' \in Ne$ such that $o(e') = v'$, and hence $o(ge') = v$ where $g \in N$ is such that $gv' = v$.  In other words, all edges of $\Gamma/N$ starting at $Nv$ are represented by edges of $\Gamma$ starting at $v$.  Hence $\deg(Nv) \le \deg(v)$, and $\deg(Nv) = \deg(v)$ if and only if every edge in $o\inv (v)$ is mapped to a distinct edge of $\Gamma/N$.  Since $v \in V\Gamma$ was arbitrary, the conclusions for the degree of $\Gamma/N$ are clear.

$(2)$ Let $H$ be the vertex stabilizer of $Nv$ in $G$; that is, $H$ is the setwise stabilizer of $Nv$. That $N$ is normal ensures $Nv$ is a block of imprimitivity for the action of $G$ on $V\Gamma$. We thus deduce that $G_{(v)} \le H$, so $G_{(v)} = H_{(v)}$.  Since $N$ is transitive on $Nv$ and $N \le H$, it follows that $NG_{(v)} = H$.
\end{proof}

\begin{defn} 
For $G$ a \tdlc group, a \defbold{Cayley--Abels} graph for $G$ is a connected graph of finite degree on which $G$ acts vertex-transitively such that the vertex stabilizers are open and compact.  A Cayley--Abels graph for a locally compact group $G$ is a Cayley--Abels graph for the \tdlc group $G/G^\circ$. That is to say, a Cayley--Abels graph for a locally compact group $G$ is a locally finite connected graph on which $G$ acts vertex-transitively and such that the vertex stabilizers are open and connected-by-compact.
\end{defn}

The following proposition is a standard result; see for example \cite[Proposition~2.E.9]{CdH14}. 

\begin{prop}\label{prop:factor} 
Let $G$ be a locally compact group. The group $G$ has a Cayley--Abels graph if and only if $G$ is compactly generated.  Moreover, if $G$ is compactly generated, then for every compact open subgroup $U/G^\circ$ of $G/G^\circ$, there exists a Cayley--Abels graph $\Gamma$ for $G$ such that $U$ is a vertex stabilizer.
\end{prop}

The vertices of a Cayley--Abels graph have the same degree, since $G$ acts vertex-transitively by graph automorphisms. This leads to an invariant for compactly generated groups.

\begin{defn}
If $G$ is a compactly generated locally compact group, the \defbold{degree} $\deg(G)$ of $G$ is the smallest degree of a Cayley--Abels graph for $G$.  
\end{defn}

We see that $\deg(G)=0$ if and only if $G$ is connected-by-compact, hence we infer the following:
\begin{obs} For $G$ a compactly generated locally compact group $\deg(G)+\dim^\infty_{\Rb}(G)=0$ 
if and only if $G$ is compact.
\end{obs}

Essential to the work at hand are the groups which act on (Cayley--Abels) graphs like discrete groups.
\begin{defn}
Given a group $G$ acting on a graph $\Gamma$, we say that $G$ acts \defbold{freely modulo kernel} on $\Gamma$ if the vertex stabilizer $G_{(v)}$ acts trivially on both the vertices and the edges of $\Gamma$ for all $v \in V$.
\end{defn}

\begin{prop}\label{prop:degree_of_quotient}
Let $G$ be a compactly generated locally compact group, $N$ be a closed normal subgroup of $G$, and $\Gamma$ be a connected graph of finite degree on which $G$ acts vertex-transitively.
\begin{enumerate}[(1)]
\item If $\Gamma$ is a Cayley--Abels graph for $G$, then $\Gamma/N$ is a Cayley--Abels graph for $G/N$. 
\item We have $\deg(\Gamma/N) \le \deg(\Gamma)$, with equality if and only if $N$ acts freely modulo kernel on $\Gamma$.
\end{enumerate}
\end{prop}

\begin{proof}
$(1)$ The graph $\Gamma/N$ is connected, and $G$ acts vertex-transitively on $\Gamma/N$. Lemma~\ref{lem:quotient:basic}(1) ensures that $\deg(\Gamma/N)$ is also finite.  The fact that the vertex stabilizers are connected-by-compact and open in $G/N$ follows from Lemma~\ref{lem:quotient:basic}(2).

$(2)$ Fix $v \in V$.  Lemma~\ref{lem:quotient:basic} ensures $\deg(\Gamma/N) \le \deg(\Gamma)$ with equality if and only if the elements of $o\inv (v)$ all lie in distinct $N$-orbits.  We thus deduce the first claim of $(2)$, and for the second, it suffices to show the elements of $o\inv (v)$ all lie in distinct $N$-orbits if and only if $N$ acts freely modulo kernel on $\Gamma$.

Suppose $N$ acts freely modulo kernel on $\Gamma$. For an edge $e$ of $\Gamma$, if $o(e)$ is fixed by $g \in N$, then $ge = e$.  The elements of $o\inv (v)$ thus all lie in distinct $N$-orbits.

Conversely, suppose the elements of $o\inv (v)$ all lie in distinct $N$-orbits. Any element of $N$ that fixes $v$ must also fix $o\inv (v)$ pointwise.  Each $g \in N_{(v)}$ then fixes $t(e)$, so $g$ fixes all the neighbors of $v$. We thus conclude $N_{(v)} \le N_{(w)}$ where $w\in V\Gamma$ is adjacent to $v$. 

As $G$ acts vertex-transitively on $\Gamma$ and $N$ is normal in $\Gamma$, the choice of $v$ is not important, so in fact, $N_{(v')} \le N_{(w')}$ for $(v',w')$ any pair of adjacent vertices of $\Gamma$.  That $\Gamma$ is connected now implies all vertex stabilizers of $N$ acting on $\Gamma$ are equal. Moreover, they fix every edge of $\Gamma$, since every edge lies in $o\inv (w)$ for some $w \in V\Gamma$.  The group $N$ therefore acts freely modulo kernel on $\Gamma$.
\end{proof}

\section{Finiteness properties of the lattice of closed normal subgroups}\label{sec:finitness_prop}
We here establish a finiteness property of the lattice of closed normal subgroups. As an immediate consequence, we obtain the existence of certain quotients with an interesting minimality condition.

\subsection{Directed and filtering families of normal subgroups}
Claim $(1)$ of the next lemma is more or less a restatement of \cite[Proposition 2.5]{CM11}. We give a proof for completeness.

\begin{lem}\label{lem:degree:finiteness}
Let $G$ be a compactly generated \tdlc group and $\Gamma$ be a Cayley--Abels graph for $G$.
\begin{enumerate}[(1)]
\item Let $\mc{F}$ be a filtering family of closed normal subgroups of $G$ and set $M := \bigcap \mc{F}$.  Then there exists $N \in \mc{F}$ such that $\deg(\Gamma/N) = \deg(\Gamma/M)$.

\item Let $\mc{D}$ be a directed family of closed normal subgroups of $G$ and set $M := \cgrp{\mc{D}}$.  Then there exists $N \in \mc{D}$ such that $\deg(\Gamma/N) = \deg(\Gamma/M)$.
\end{enumerate}
\end{lem}

\begin{proof}
Fix $v \in V\Gamma$. For $N \in \Norm(G)$, the value $\deg(\Gamma/N)$ is determined by the number of orbits of the action of $N_{(v)}$ on the set $X$ of edges issuing from $v$.  Defining $\alpha(N)$ to be the subgroup of $\Sym(X)$ induced by $N_{(v)}$ on $X$, we see that $\deg(\Gamma/N)=\deg(\Gamma/M)$ if $\alpha(N)=\alpha(M)$. The assignment $N \mapsto \alpha(N)$ is also order-preserving. For a filtering or directed family $\mc{N} \subseteq \Norm(G)$, the family $\alpha(\mc{N}):= \{\alpha(N) \mid N \in \mc{N}\}$ is then a filtering or directed family of subgroups of $\Sym(X)$.  That $\Sym(X)$ is a finite group ensures $\alpha(\mc{N})$ is a finite family, so it admits a minimum or maximum, according to whether $\mc{N}$ is filtering or directed.

Claim $(2)$ is now immediate. The directed family $\alpha(\mc{D})$ admits a maximal element $\alpha(N)$. It then follows that $\alpha(N)=\alpha(M)$, so $\deg(\Gamma/N)=\deg(\Gamma/M)$.

For $(1)$, an additional compactness argument is required. If $G$ acts freely modulo kernel on $\Gamma$, then $N\in \mc{F}$ and $M$ also act as such. The desired result then follows since $\deg(\Gamma/N)=\deg(\Gamma)=\deg(\Gamma/M)$. Let us assume that $G$ does not act freely modulo kernel, so $G_{(v)}$ acts non-trivially on $\Gamma$ for any $v\in V\Gamma$. 

Take $\alpha(N) \in \alpha(\mc{F})$ to be the minimum.  Given $r \in \alpha(N)$, let $Y$ be the set of elements of $G_{(v)}$ that \textit{do not} induce the permutation $r$ on $X$. If $r\neq 1$, then plainly $Y\neq G_{(v)}$. If $r=1$, then $Y\neq G_{(v)}$ since $G_{(v)}$ acts non-trivially on $\Gamma$. The set $Y$ is a proper open subset of $G_{(v)}$, and thus $G_{(v)} \setminus Y$ is a non-empty compact set. 

Letting $\mc{K}$ be a finite subset of $\mc{F}$, the group $K := \bigcap_{F \in \mc{K}}F$ contains some element of $\mc{F}$, so $\alpha(K) \ge \alpha(N)$. In particular, $K_{(v)} \not\subseteq Y$. The intersection
\[
\bigcap_{F \in \mc{K}}(F_{(v)} \cap (G_{(v)} \setminus Y))
\]
is therefore non-empty.  Compactness now implies that
\[
M_{(v)} \cap (G_{(v)} \setminus Y) = \bigcap_{F \in \mc{F}}(F_{(v)} \cap (G_{(v)} \setminus Y)) \not= \emptyset;
\]
that is, some element of $M_{(v)}$ induces the permutation $r$ on $X$. Since $r \in \alpha(N)$ is arbitrary, we conclude that $\alpha(M) = \alpha(N)$, and so $\deg(\Gamma/N) = \deg(\Gamma/M)$.
\end{proof}

The conclusion of claim $(1)$ in Lemma~\ref{lem:degree:finiteness} implies that the factor $N/M$ is ``discrete" from the point of view of the Cayley--Abels graph.

\begin{lem}\label{lem:ess_discrete_char}
Let $G$ be a compactly generated \tdlc group with $N$ a closed normal subgroup of $G$.  If there is a Cayley--Abels graph $\Gamma$ for $G$ such that $\deg(\Gamma/N) = \deg(\Gamma)$, then there exists a compact normal subgroup $L$ of $G$ acting trivially on $\Gamma$ such that $L$ is an open subgroup of $N$.
\end{lem}
\begin{proof}
In view of Proposition~\ref{prop:degree_of_quotient}, $N$ acts freely modulo kernel on $\Gamma$.  For $U$ the pointwise stabilizer of the star $o^{-1}(v)$ for some vertex $v$, the subgroup $U$ is a compact open subgroup of $G$, and its core $K$ is the kernel of the action of $G$ on $\Gamma$.  Since $N$ acts freely modulo kernel, we deduce that $N \cap U \leq K$. The group $L:=K\cap N$ now satisfies the lemma.
\end{proof}

Combining our results on Cayley--Abels graphs with the Gleason--Yamabe Theorem, we obtain a result that applies to compactly generated locally compact groups without dependence on a choice of Cayley--Abels graph.

\begin{thm}\label{thm:essential_finiteness}Let $G$ be a compactly generated locally compact group.
\begin{enumerate}[(1)]
\item If $\mc{F}$ is a filtering family of closed normal subgroups of $G$, then there exists $N \in \mc{F}$ and a closed normal subgroup $K$ of $G$ such that $\bigcap \mc{F} \le K \le N$, $K/\bigcap \mc{F}$ is compact, and $N/K$ is discrete.

\item If $\mc{D}$ is a directed family of closed normal subgroups of $G$, then there exists $N \in \mc{D}$ and a closed normal subgroup $K$ of $G$ such that $N \le K \le \cgrp{\mc{D}}$, $K/N$ is compact, and $\cgrp{\mc{D}}/K$ is discrete.
\end{enumerate}
\end{thm}

\begin{proof}
$(1)$  The group $G/\bigcap \mc{F}$ is a compactly generated locally compact group, so we assume that $\bigcap \mc{F}=\triv$.  Fix a Cayley--Abels graph $\Gamma$ for $G$ and let $E$ be the kernel of the action of $G$ on $\Gamma$.  Since $G^\circ \le E$, we have $G^\circ = E^\circ$, and $E/G^\circ$ is compact, since $E/G^\circ$ is the core of a compact open subgroup of $G/G^\circ$.  Theorem~\ref{thm:yamabe_radical} thus ensures $R: = \Rad{\mc{LE}}(E)$ is compact and the quotient $E/R$ is a Lie group with finitely many connected components. Observe additionally that $R\normal G$.

For $N \in \Norm(G)$, set $a_N: = \deg(\Gamma/\ol{NG^{\circ}})$ and $b_N := \dim_{\Rb}((N \cap E)R/R)$.  Both $a_N$ and $b_N$ are natural numbers depending on $N$ in a monotone fashion: given $N,N' \in \Norm(G)$ such that $N \le N'$, then $a_N \ge a_{N'}$ and $b_N \le b_{N'}$. By Lemma~\ref{lem:degree:finiteness}, there exists $N \in \mc{F}$ such that $a_N=\deg(\Gamma/\ol{NG^{\circ}}) = \deg(\Gamma)$.  Since $\mc{F}$ is a filtering family, we can choose $N$ such that additionally $b_N$ is minimized across $N \in \mc{F}$.  Fix such an $N$. 

We argue claim $(1)$ holds for $N$ and $K:=N\cap E\cap R$. Since $R$ is compact, it suffices to show $N/K$ is discrete. Consider first $ N \cap E$. By Proposition~\ref{prop:degree_of_quotient}, $\ol{NG^{\circ}}$ acts freely modulo kernel on $\Gamma$, hence $N$ acts freely modulo kernel on $\Gamma$.  The group $N \cap E$ is thus a vertex stabilizer of the action of $N$ on $\Gamma$. In particular, $N \cap E$ is open in $N$.  It now suffices to show that $K$ is open in $N \cap E$.

Consider $((N \cap E)R/R)^\circ$ and suppose for contradiction that there is an infinite compact identity neighborhood $T/R$ of $((N \cap E)R/R)^{\circ}$; we assume that $TR=T$. Find $U \varsubsetneq T$ open containing $1$ with $U = UR$. For $F \in \mc{F}$ with $F \le N$, the minimality of $b_N$ implies $(F \cap E)R/R$ is a subgroup of $(N \cap E)R/R$ with the same dimension $b_N$. Consequently, $(F \cap E)R/R$ contains $((N \cap E)R/R)^\circ$, and
\[
T = FR \cap T = (F \cap T)R.
\] 
We infer that $(F \cap T) \not\subseteq U$, so $F$ intersects the non-empty compact set $T\setminus U$.  As $\mc{F}$ is a filtering family, it follows by compactness that $\bigcap\mc{F} \cap T$ intersects $T\setminus U$. However, this is absurd since $\bigcap\mc{F} \cap T=\triv$.  

All compact identity neighborhoods of $((N \cap E)R/R)^\circ$ are thus finite. It follows that a singleton set is open, so $((N \cap E)R/R)^\circ$  is discrete. The only discrete connected group is the trivial group, so  $((N \cap E)R/R)^\circ$ is in fact trivial. Since $(N \cap E)R/R$ is a Lie group, we conclude the group $(N \cap E)R/R\simeq (N \cap E)/K$ is discrete, verifying $(1)$ holds for $N$ and $K$.

\

$(2)$ Set $M:=\cgrp{\mc{D}}$ and let $\Gamma$ be a Cayley--Abels graph for $G$.  By Lemma~\ref{lem:degree:finiteness}, there exists $N \in \mc{D}$ such that $\deg(\Gamma/\ol{NG^{\circ}}) = \deg(\Gamma/\ol{MG^{\circ}})$. Moreover, $\Gamma/\ol{NG^{\circ}}=\Gamma/N$ is a Cayley--Abels graph for $G/N$ by Proposition~\ref{prop:degree_of_quotient}.  Passing to the quotient $G/N$ and replacing $\Gamma$ with $\Gamma/N$, we may assume that $\deg(\Gamma) = \deg(\Gamma/\ol{MG^{\circ}})$ and that $M$ acts freely modulo kernel on $\Gamma$.  

Let $E$ be the kernel of the action of $G$ on $\Gamma$ and let $L := M \cap E$. Our assumptions ensure that $L$ is open in $M$. We will now find an $F\in \mc{D}$ and a closed $J\normal G$ such that $J$ is an open subgroup of $L$ and that $JF/F$ is compact; this will prove $(2)$ with $K:=JF$ and $N:=F$.

As in $(1)$, the group $R: = \Rad{\mc{LE}}(E)$ is compact, and the quotient $E/R$ is a Lie group with finitely many connected components.  Let $N \in \mc{D}$ witness the minimum of the set
\[
\{\dim_{\Rb}(E/(N \cap E)R) \mid N \in \mc{D}\}
\]
and consider $S:=(N \cap E)R$.  For all $D \in \mc{D}$ such that $D \ge N$, the quotient $(D \cap E)R/S$ is discrete, by our choice of $S$.  Every element of $(D \cap E)R/S$ thereby has an open centralizer in $G$, and hence $(D \cap E)R/S$ is centralized by $(G/S)^\circ$. As $L=M\cap E$ is open in $M$,
\[
L = \ol{\bigcup_{D \in \mc{D}}(D \cap M\cap E)} = \ol{\bigcup_{D \in \mc{D}}(D \cap E)}.
\] 
The group $LR/S$ is thus centralized by $(G/S)^\circ$, and a fortiori, $(LR/S)^\circ$ is abelian.  

The group $E/S$ is a Lie group,  so $LR/S$ is also a Lie group. We infer that $(LR/S)^\circ$ is an open subgroup of $LR/S$. The component $(LR/S)^\circ$ thus contains a dense subgroup which is formed of a directed union of discrete subgroups of the form $(D \cap E)R/S \cap (LR/S)^\circ$. Since $(LR/S)^\circ$ is a connected abelian Lie group, Lemma~\ref{connected_abelian_approximation} ensures there is some $F \in \mc{D}$ such that $(F \cap E)R/S \cap (LR/S)^\circ$ is cocompact in $(LR/S)^\circ$. Fix $F\in\mc{D}$ with this property. 

There is a compact set $A\subseteq LR/S$ such that $(LR/S)^\circ \subseteq A(F\cap E)R/S$. Letting $\pi:LR/S\rightarrow LR/(F\cap E)R$ be the usual projection, $\pi((LR/S)^\circ )\leq \pi(A)$ which is compact.  Let $J\normal L$ be the preimage in $L$ of the connected component of $L/(F\cap E)(R\cap L) \simeq LR/(F \cap E)R$.   We see that $J/(F \cap E)(R \cap L)$ is compact, and since $R$ is compact, in fact $J/(F \cap E)$ is compact.

The subgroup $J$ is invariant under continuous automorphisms of $L$ that fix $(F\cap E)(R\cap L)$, so $J\normal G$. Additionally, $J$ is an open subgroup of $L$. Forming $JF$, the quotient $JF/F\simeq J/F\cap J$ is a quotient of the compact group $J/(F \cap E)$. We have thus found the desired subgroups, and $(2)$ follows.
\end{proof}

\subsection{Just-non-${P}$ quotients}
As an immediate application of Theorem~\ref{thm:essential_finiteness}, we show certain quotients exist.
\begin{defn}
For $P$ a property of groups, we say that a non-trivial locally compact group $G$ is \defbold{just-non-$P$} if every proper non-trivial quotient of $G$ has $P$, but $G$ itself does not have $P$.
\end{defn}

The properties $P$ we consider must enjoy the following permanence property: A property $P$ of locally compact groups is \defbold{closed under compact extensions} if for a compactly generated locally compact group $G$ with a compact normal subgroup $N$, the group $G$ has $P$ if and only if $G/N$ has $P$. Quasi-isometry invariant properties are closed under compact extensions; see \cite[Proposition 1.D.4]{CdH14}.

A compactly generated locally compact group $G$ is called \defbold{compactly presented} if there is a compact generating set $S$ such that $G$ has a presentation $\grp{S|R}$ where the relators $R$ have bounded length. Being compactly presented is a quasi-isometry invariant property of compactly generated locally compact groups; see \cite[Corollary 8.A.4]{CdH14}.   

\begin{thm}\label{thm:just-non-P}
Let $P$ be a property of locally compact groups.  If all groups with $P$ are compactly presented and $P$ is closed under compact extensions, then for any compactly generated locally compact group $G$, exactly one of the following holds:
\begin{enumerate}[(1)]
\item Every non-trivial quotient of $G$ (including $G$ itself) has $P$; or 
\item $G$ admits a quotient that is just-non-$P$.
\end{enumerate}
\end{thm}

\begin{proof}
Let $\mc{F}$ be the collection of proper closed normal subgroups $N$ of $G$ such that $G/N$ fails to have property $P$. If $\mc{F}=\emptyset$, then $G$ satisfies (1), and we are done. Assuming $\mc{F}\neq \emptyset$, we claim increasing chains in $\mc{F}$ have upper bounds. Let $(N_{\alpha})_{\alpha\in I}$ be an $\subseteq$-increasing chain and put $L:=\ol{\bigcup_{\alpha \in I}N_{\alpha}}$. Suppose for contradiction $G/L$ has property $P$; in particular, $G/L$ is compactly presented.

Appealing to Theorem~\ref{thm:essential_finiteness}, we may find $\gamma\in I$ and a closed $K\normal G$ such that $N_{\gamma}\leq K\leq L$ with $L/K$ discrete and $K/N_{\gamma}$ compact. The group $G/L$ is a quotient of $G/K$ with kernel $L/K$. Moreover, since $G/L$ is compactly presented, \cite[Proposition 8.A.10]{CdH14} implies the discrete group $L/K$ is finitely normally generated as a subgroup of $G/K$.

Let $X\subseteq L/K$ be a finite normal generating set. Since $\bigcup_{\alpha \in I}N_{\alpha}$ is dense in $L$ and $K$ is open in $L$, we may find $N_{\beta}$ for some $\beta>\gamma$ such that $N_{\beta}K/K$ contains $X$. The normality of $N_{\beta}$ implies that indeed $N_{\beta}K=L$. The group $G/L$ is thus a quotient of $G/N_{\beta}$ with kernel $KN_{\beta}/N_{\beta}$. The kernel $KN_{\beta}/N_{\beta}$ is compact, so $G/N_{\beta}$ is a compact extension of $G/L$. Hence, $G/N_{\beta}$ has property $P$, an absurdity. 

We conclude that increasing chains in $\mc{F}$ have upper bounds. Applying Zorn's lemma, we may find $N\in \mc{F}$ maximal. The maximality of $N$ ensures every proper non-trivial quotient of $G/N$ has property $P$, verifying $(2)$.
\end{proof}

Let $G$ be a compactly generated locally compact group with compact symmetric generating set $X$. The group $G$ has \textbf{polynomial growth} if there are constants $C,k>0$ such that $\mu(X^n)\leq Cn^k$ for all $n\geq 1$, where $\mu$ is a Haar measure. Having polynomial growth is a quasi-isometry invariant property, and compactly generated locally compact groups with polynomial growth are \textit{necessarily} compactly presented, \cite[Proposition 8.A.25]{CdH14}. 

The following corollary is now immediate from Theorem~\ref{thm:just-non-P}.

\begin{cor} Let $G$ be a compactly generated locally compact group.
\begin{enumerate}[(1)]
\item If some quotient of $G$ is not compactly presented, then $G$ admits a quotient that is just-non-(compactly presented).
\item If $G$ does not have polynomial growth, then $G$ admits a quotient that is just-not-(of polynomial growth).
\end{enumerate}
\end{cor}

We remark that groups in which all proper quotients are compactly presented satisfy a stronger version of Theorem~\ref{thm:essential_finiteness}(2).

\begin{prop}Let $G$ be a compactly generated locally compact group such that every proper quotient of $G$ is compactly presented. If $\mc{D}$ is a directed family of closed normal subgroups of $G$, then there exists $N \in \mc{D}$ such that $\cgrp{\mc{D}}/N$ is compact.
\end{prop}

\begin{proof}
Without loss of generality $\cgrp{\mc{D}} \neq \triv$.  By Theorem~\ref{thm:essential_finiteness}(2), there exists  $M \in \mc{D}$ and a closed normal subgroup $K$ of $G$ such that $M \le K \le \cgrp{\mc{D}}$, $\cgrp{\mc{D}}/K$ is discrete, and $K/M$ is compact.  The group $G/\cgrp{\mc{D}}$ is compactly presented, so $\cgrp{\mc{D}}/K$ is normally generated in $G$ by a finite set. It now follows there exists $M\leq N \in \mc{D}$ such that $\cgrp{\mc{D}} = NK$.  In particular, $\cgrp{\mc{D}}/N$ is compact, as required.
\end{proof}

\section{Essentially chief series}\label{sec:chief}

\begin{defn}
An \defbold{essentially chief series} for a compactly generated locally compact group $G$ is a finite series
\[
\{1\} = G_0 \leq G_1 \leq \dots \leq G_n = G
\]
of closed normal subgroups such that each normal factor $G_{i+1}/G_i$ is either compact, discrete, or a topological chief factor of $G$.
\end{defn}

We now show that any compactly generated locally compact group admits an essentially chief series; more precisely, any finite normal series can be refined to an essentially chief series.

\subsection{Existence of essentially chief series} 
We begin with two technical results which prove the existence of essentially chief refinements of normal series, with bounds on the number of factors required. These lemmas deal with the connected case and totally disconnected case, respectively.

\begin{lem}\label{lem:chief_refinement:connected}
Suppose that $G$ is a compactly generated locally compact group, $H\leq L$ are closed normal subgroups of $G$, and $d := \dim_{\Rb}^{\infty}(L)-\dim_{\Rb}^{\infty}(H)$. If $L/H$ is connected-by-compact, then there is a series
\[
H = G_0 \leq G_1 \leq \dots \leq G_k = L
\]
of closed normal subgroups of $G$ with $k\leq 2d+1$ such that each factor $G_{i+1}/G_i$ is either compact, discrete, or a chief factor of $G$.  Additionally, at most $d$ factors are neither compact nor discrete.\end{lem}

\begin{proof}
Let us first assume that $L/H$ is connected and $\Rad{\mc{LE}}(L/H) = \triv$; we will here prove the result with $k \le 2d$.  In this situation, $L/H$ is a Lie group and $d = \dim_{\Rb}(L/H) $.  There is then $i \le d$ and a series $H = M_0 < M_1 \dots < M_i = L$ of closed $G$-invariant subgroups such that each of the factors $M_j/M_{j-1} =: V_j$ is connected, has positive dimension, and has no proper closed $G$-invariant subgroup of positive dimension.  

For each $1 \le j \le i$, the factor $V_j$ is either compact, a non-compact semisimple Lie group such that $G$ permutes transitively the simple factors, or $\Rb^n$ such that $G$ acts irreducibly. If $V_j$ is compact, we do nothing. In the second case, we take $N_j$ such that $N_j/M_{j-1} = \Z(V_j)$. The factor $N_j/M_{j-1}$ is then discrete, and $M_j/N_j$ is a chief factor of $G$. For the last case, every proper closed $G$-invariant subgroup of $V_j$ is either trivial or a lattice.  If $G$ does not preserve a lattice, then $V_j$ is already a chief factor.  If $G$ preserves a lattice $N_j/M_{j-1}$, then $N_j/M_{j-1}$ is discrete, and $M_j/N_j$ is compact.

By including the $N_j$ terms as needed, we obtain a $G$-invariant series from $H$ to $L$ with at most $k  \le 2d$ factors such that each factor is compact, discrete, or a chief factor. Additionally, at most $d$ of the factors are neither compact nor discrete, since each $V_j$ contributes at most one such factor. Let us make a further observation for later use: If there are exactly $2d$ factors in the series, then each $V_j$ is one-dimensional, hence abelian, and divided into a discrete factor $N_j/M_{j-1}$ and a compact factor $M_j/N_j$.

For the general case, let $L'$ be the preimage of $(L/H)^\circ$ and let $H'$ be the preimage of $\Rad{\mc{LE}}(L'/H)$.  By our work above, there is a $G$-invariant series from $H'$ to $L'$ with at most $2d$ factors of the appropriate form. Since both $H'/H$ and $L/L'$ are compact, we immediately obtain the required $G$-invariant series from $H$ to $L$ with at most $l+2$ factors where $l\leq 2d$ such that at most $d$ factors are non-compact.  If $l = 2d$, the uppermost factor in the series from $H'$ to $L'$ can be combined with $L/L'$, so we find a series with $2d+1$ factors. We thus deduce that there is an essentially chief series from $H$ to $L$ with at most $2d+1$ factors such that at most $d$ factors are non-compact.
\end{proof}

\begin{lem}\label{lem:chief_refinement:tdlc}
Suppose that $G$ is a compactly generated locally compact group, $H\leq L$ are closed normal subgroups of $G$, and $\Gamma$ is a Cayley--Abels graph for $G$ such that $\deg(G)=\deg(\Gamma)$. If $G^{\circ}\leq H$, then there exists a series
\[
H=:C_0\leq K_0\leq D_0\leq \dots \leq C_n\leq K_n\leq  D_n=L
\]
of closed normal subgroups of $G$ with $n\leq \deg(\Gamma/H)-\deg(\Gamma/L)$ such that
\begin{enumerate}[(1)]
\item for $0\leq i\leq n$, $K_i/C_{i}$ is compact, and $D_i/K_i$ is discrete; and
\item for $1\leq i\leq n$, $C_{i}/D_{i-1}$ is a chief factor of $G$.
\end{enumerate}
\end{lem}
\begin{proof}
Set $k:=\deg(\Gamma/H)$ and $m:=\deg(\Gamma/L)$. By induction on $i\leq k-m$, we prove there exists a series of normal subgroups of $G$
\[
H=:C_0\leq K_0\leq D_0\leq \dots \leq C_i\leq K_i\leq  D_i\leq L
\]
such that claims $(1)$ and $(2)$ hold of all factors up to $i$ and that there is $i\leq j \leq k-m$ for which $D_i$ is maximal among normal subgroups of $G$ such that $\deg(\Gamma/D_i)=k-j$ and $D_i\leq L$. 

For $i=0$, it follows from Lemma~\ref{lem:degree:finiteness} and Zorn's lemma that there exists $D_0$ maximal such that $\deg(\Gamma/D_0)=k$ and $H\leq D_0\leq L$. The graph $\Gamma/H$ is a Cayley--Abels graph for $G/H$ with degree $k$, and 
\[
\deg((\Gamma/H)/(D_0/H))=\deg(\Gamma/D_0)=k.
\]  
Applying Lemma~\ref{lem:ess_discrete_char}, there is $K_0\normal G$ such that $H\leq K_0\leq D_0$ with $K_0/H$ compact, open, and normal in $D_0/H$. We conclude that $C_0=H$, $K_0$, and $D_0$ satisfy the inductive claim when $i=0$ with $j=0$.

Suppose we have built our sequence up to $i$. By the inductive hypothesis, there is $i\leq j \leq k-m$ such that $D_i$ is maximal with $\deg(\Gamma/D_i)=k-j$ and $D_i\leq L$. If $j=k-m$, then the maximality of $D_i$ implies $D_i=L$, and we stop. Else, let $j'>j$ be least such that there is $M\normal G$ with $\deg(\Gamma/M)=k-j'$ and $D_i\leq M \leq L$. We take $C_{i+1}\normal G$ to be minimal such that $\deg(\Gamma/C_{i+1})=k-j'$ and $D_i< C_{i+1}\leq L$; Lemma~\ref{lem:degree:finiteness} ensures such a subgroup exists.

Consider a closed $N\normal G$ with $D_i\leq N<C_{i+1}$. Putting $\deg(\Gamma/N)=l$, Proposition~\ref{prop:degree_of_quotient} implies $k-j'\leq l \leq k-j$, and the minimality of $C_{i+1}$ further implies $k-j'<l$. On the other hand, we chose $j'>j$ least such that there is $M\normal G$ with $\deg(\Gamma/M)=k-j'$ and $D_i\leq M\leq L$. Therefore, $l=k-j$. In view of the maximality of $D_i$, we deduce that $D_i=N$ and that $C_{i+1}/D_{i}$ is a chief factor.

Applying again Lemma~\ref{lem:degree:finiteness}, there is $D_{i+1}\normal G$ maximal such that 
\[
\deg(\Gamma/D_{i+1})=k-j'
\]
and $C_{i+1}\leq D_{i+1}\leq L$. Lemma~\ref{lem:ess_discrete_char} supplies $K_{i+1}\normal G$ such that $C_{i+1}\leq K_{i+1}\leq D_{i+1}$ with $K_{i+1}/C_{i+1}$ compact and open in $D_{i+1}/C_{i+1}$. This completes the induction.

Our procedure halts at some $n\leq k-m$. At this stage, $D_n=L$, verifying the theorem.
\end{proof}

We now use Lemmas~\ref{lem:chief_refinement:connected} and \ref{lem:chief_refinement:tdlc} to refine a normal series factor by factor to produce an essentially chief series.

\begin{thm}\label{thm:chief_series}
Let $G$ be a compactly generated locally compact group and let $(G_i)_{i=1}^{m-1}$ be a finite ascending sequence of closed normal subgroups of $G$.  Then there exists an essentially chief series for $G$
\[
\triv = K_0 \le K_1 \le \dots \le K_l = G,
\]
such that
\begin{enumerate}[(1)]
\item $\{G_1,\dots,G_{m-1}\}$ is a subset of $\{K_0,\dots,K_l\}$; and
\item if $G^\circ \in \{G_1,\dots,G_{m-1}\}$, then $l \le 2m + 2\dim^\infty_{\Rb}(G)+3\deg(G)$, and at most $\dim^\infty_{\Rb}(G)+\deg(G)$ of the factors $K_{i+1}/K_i$ are neither compact nor discrete.
\end{enumerate}
\end{thm}

\begin{proof} 
Let us extend the series by $G_0:=\{1\}$ and $G_m:=G$ obtaining the series
\[
\triv = :G_0 \le G_1 \le \dots \le G_{m-1} \le G_{m} := G.
\]

For each $j \in \{0,\dots,m-1\}$, we apply Lemma~\ref{lem:chief_refinement:connected} to $H := G_j$ and $L$ such that $L/G_j = (G_{j+1}/G_j)^\circ$ to refine the series. We then refine again by applying Lemma~\ref{lem:chief_refinement:tdlc} to $L := G_{j+1}$ and $H$ such that $H/G_j = (G_{j+1}/G_j)^\circ$. This yields the desired refined series claimed in $(1)$.

For $(2)$, suppose $G^\circ \in \{G_0,\dots,G_{m-1}\}$; say that $G_k=G^{\circ}$ for some $0\leq k\leq m-1$. For each $i<k$, Lemma~\ref{lem:sgrp_con} yields a closed normal $H_{i+1}$ of $G$ with $G_i\leq H_{i+1}\leq G_{i+1}$ such that $G_{i+1}/H_{i+1}$ is discrete and $H_{i+1}/G_i$ is connected-by-compact. We may thus apply Lemma~\ref{lem:chief_refinement:connected} to each of the pairs $G_{i}\leq H_{i+1}$ with $i< k$ to produce a refined series.  (If $k=0$, there is nothing to do at this stage.)

In the refined series, there are at most $2(\dim_{\Rb}^{\infty}(G_{i+1})-\dim_{\Rb}^{\infty}(G_{i}))+2$ terms $T$ such that $G_{i} \le T < G_{i+1}$. Additionally, at most $\dim_{\Rb}^{\infty}(G_{i+1})-\dim_{\Rb}^{\infty}(G_{i})$ factors are neither compact nor discrete. The number of terms in our refined series strictly below $G_k$ is thus at most
\[
\sum_{i=0}^{k-1} \left( 2(\dim_{\Rb}^{\infty}(G_{i+1})-\dim_{\Rb}^{\infty}(G_{i}))+2 \right)=2k+2\dim_{\Rb}^{\infty}(G),
\]
and the total number of factors that are neither compact nor discrete is at most
\[
\sum_{i=0}^{k-1}(\dim_{\Rb}^{\infty}(G_{i+1})-\dim_{\Rb}^{\infty}(G_{i}))=\dim_{\Rb}^{\infty}(G).
\]

We now consider the series $G_k\leq \dots G_{m-1}\leq G_m= G$. If $G=G_k$, we are done; we thus suppose that $G_k<G$. Let $\Gamma$ be a Cayley--Abels graph for $G$ with $\deg(\Gamma)=\deg(G)$ and put $k_j := \deg(\Gamma/G_j)$.  For each $j \in \{k,\dots,m-1\}$, we apply Lemma~\ref{lem:chief_refinement:tdlc} to each pair $G_j\leq G_{j+1}$ to obtain a refined series.  This results in a series in which the number of terms $T$ such that $G_j \le T < G_{j+1}$ is at most $3(k_j - k_{j+1}) +2$, and at most $k_j - k_{j+1}$ of the factors are neither compact nor discrete. The total number of terms in the refined series not including $G_{m}$ is thus at most
\[
\begin{array}{ccl}
\sum_{j=k}^{m-1}(3(k_j-k_{j+1}) + 2) &  = & 2(m-k) +3(\deg(\Gamma)-\deg(\Gamma/G))\\
									& \leq&  2(m-k) +3\deg(G),
\end{array}
\]
and the total number of non-compact, non-discrete factors is at most
\[
\sum_{j=1}^{m-1}(k_j-k_{j+1}) \leq \deg(G). 
\]

Putting together our two refined series, we obtain an essentially chief series
\[
\triv=K_0\leq K_1\leq\dots\leq K_l=G
\]
such that 
\[
\begin{array}{ccl}
l & \leq & 2k+2\dim_{\Rb}^{\infty}(G)+ 2(m-k) +3\deg(G)\\
  & \leq & 2m+2\dim_{\Rb}^{\infty}(G)+3\deg(G).
  \end{array}
\]
Furthermore, the number of factors that are neither compact nor discrete is at most $\dim_{\Rb}^{\infty}(G)+\deg(G)$.
\end{proof}

\begin{rmk} For any $n\geq 0$, the direct product of $n$ compactly generated non-discrete simple locally compact groups gives a compactly generated \lcsc group such that any essentially chief series has $n$ chief factors. 

One can also construct compactly generated \lcsc groups such that any essentially chief series has at least $n$ compact factors or $n$ discrete factors. This can be arranged by taking wreath products. For example, let $K$ be a non-abelian finite simple group and $V$ be an infinite finitely generated simple group. The group $V$ acts on $K^V$, so we may form $G:=K^V\rtimes V$, which is locally compact. Fixing $U$ a proper open subgroup of $K^V$, $G$ acts faithfully and transitively on the cosets $X:=G/U$. We obtain a second locally compact group $G':=\bigoplus_X V\rtimes G$. One checks that any chief series for $G$ has one compact factor and two discrete factors. By taking more wreath products, one produces any desired number of discrete or compact factors.
\end{rmk}

\subsection{Uniqueness of essentially chief series}

We finally obtain a Jordan--H\"{o}lder theorem for essentially chief series, using the general properties of associated chief factors obtained in \cite{RW_P_15}.  For this uniqueness result, we need to exclude chief factors associated to compact and discrete factors. 

\begin{defn} 
For $G$ a Polish group and $K/L$ a chief factor of $G$, we say that $K/L$ is \textbf{negligible} if $K/L$ is either abelian or associated to a compact or discrete chief factor. A chief block $\mf{a}\in \mf{B}_G$ is \textbf{negligible} if $\mf{a}$ has a compact or discrete representative. The collection of non-negligible chief blocks of $G$ is denoted by $\mf{B}_G^{*}$. See Subsection~\ref{sec:chief_factor} for the definitions of the association relation and chief blocks.
\end{defn}

Using the general methods outlined in \cite[Appendix II]{CM11}, one can produce negligible chief factors which are neither abelian, compact, nor discrete.
  
\begin{rmk} In \cite{RW_LC_16}, we show that all negligible chief factors $K/L$ in an \lcsc group are either quasi-discrete, meaning that the elements of $K/L$ with open centralizer form a dense subgroup of $K/L$, or compact. The work \cite{CM11} furthermore shows quasi-discrete groups have restrictive topological structure.
\end{rmk} 

In contrast to the results about existence of chief series, we do not need to assume that $G$ is compactly generated for our Jordan--H\"{o}lder theorem.

\begin{thm}\label{thm:association_map}
Suppose that $G$ is an \lcsc group and that $G$ has two essentially chief series $(A_i)_{i=0}^m$ and $(B_j)_{j=0}^n$. Define
\begin{align*}
I &:= \{ i \in \{1,\dots,m\} \mid A_i/A_{i-1} \text{ is a non-negligible chief factor of $G$}\}; and \\
J &:= \{ j \in \{1,\dots,n\} \mid B_j/B_{j-1} \text{ is a non-negligible chief factor of $G$}\}.
\end{align*}
Then there is a bijection $f:I \rightarrow J$, where $f(i)$ is the unique element $j \in J$ such that $A_i/A_{i-1}$ is associated to $B_j/B_{j-1}$. 
\end{thm}

\begin{proof}Theorem~\ref{thmintro:Schreier_refinement} provides a function $f: I \rightarrow \{1,\dots,n\}$ where $f(i)$ is the unique element of $\{1,\dots,n\}$ such that $A_i/A_{i-1}$ is associated to a subquotient of $B_{f(i)}/B_{f(i)-1}$.

If $B_{f(i)}/B_{f(i)-1}$ is compact or discrete, then $A_i/A_{i-1}$ is associated to a compact or discrete factor of $G$, which contradicts our choice of $I$.  The factor $B_{f(i)}/B_{f(i)-1}$ is thus chief, and $A_i/A_{i-1}$ is associated to $B_{f(i)}/B_{f(i)-1}$. Theorem~\ref{thmintro:Schreier_refinement} implies $B_{f(i)}/B_{f(i)-1}$ is also non-abelian. Since association is an equivalence relation for non-abelian chief factors, we conclude that $B_{f(i)}/B_{f(i)-1}$ is non-negligible, and therefore, $f(i) \in J$. 

We thus have a well-defined function $f:I \rightarrow J$.  The same argument with the roles of the series reversed produces a function $f':J \rightarrow I$ such that $B_j/B_{j-1}$ is associated to $A_{f'(j)}/A_{f'(j)-1}$.  Since each factor of the first series is associated to at most one factor of the second by Theorem~\ref{thmintro:Schreier_refinement}, we conclude that $f'$ is the inverse of $f$, hence $f$ is a bijection.
\end{proof}

\begin{cor}\label{cor:nsdim}
If $G$ is a compactly generated \lcsc group, then each $\mf{a}\in \mf{B}_G^*$ is represented exactly once in every essentially chief series for $G$, and $|\mf{B}_G^*|\leq \dim^\infty_{\Rb}(G)+\deg(G)$. \end{cor}
\begin{proof}
Let $(G_i)_{i=0}^n$ be an essentially chief series for $G$. For $\mf{a}\in \mf{B}_G^*$, fix a representative $A/B\in \mf{a}$ and use Theorem~\ref{thm:chief_series} to refine the series $\triv\leq A<B\leq G$ to a chief series $(H_i)_{i=0}^k$. Theorem~\ref{thm:association_map} now implies there is a unique $0\leq i<n$ such that $A/B$ is associated to $G_{i+1}/G_i$. Hence, $G_{i+1}/G_i\in \mf{a}$. On the other hand, association is an equivalence relation, so the uniqueness of $i$ implies that $G_{i+1}/G_i$ is the only representative of $\mf{a}$ appearing in $(G_i)_{i=0}^n$. The chief block $\mf{a}\in \mf{B}_G^*$ is thus represented exactly once in any essentially chief series for $G$.

For the second claim, we use Theorem~\ref{thm:chief_series} to produce $(K_i)_{i=0}^m$ an essentially chief series for $G$ that refines the series $\triv\leq G^{\circ}\leq G$. By the previous paragraph, each $\mf{a}\in \mf{B}_G^{*}$ admits exactly one representative with the form $K_{i+1}/K_i$ for some $0\leq i<m$. Moreover, such a representative must be neither compact nor discrete. Theorem~\ref{thm:chief_series} ensures that we can choose $(K_i)_{i=0}^m$ such that the number of non-compact, non-discrete factors is at most $\dim^\infty_{\Rb}(G)+\deg(G)$, hence $|\mf{B}_G^*|\leq \dim^\infty_{\Rb}(G)+\deg(G)$.
\end{proof}

To conclude this section, we note that each non-negligible block of a compactly generated group admits a unique smallest closed normal subgroup covering it; this somewhat technical observation is important in \cite{RW_LC_16}. See Subsection~\ref{sec:chief_factor} for the definition of a normal subgroup covering a chief block.

\begin{prop}\label{prop:non-neg:min_covered}
Let $G$ be a compactly generated \lcsc group with $\mf{a} \in \mf{B}_G^*$.  Then there is a closed normal subgroup $G_{\mf{a}}$ of $G$ such that for every closed normal subgroup $K$ of $G$, $K$ covers $\mf{a}$ if and only if $K \ge G_{\mf{a}}$.
\end{prop}

\begin{proof}Let $\mc{K}$ be the set of closed normal subgroups of $G$ that cover $\mf{a}$ and set $G_{\mf{a}}:=\bigcap \mc{K}$.  By \cite[Lemma~7.10]{RW_P_15}, the set $\mc{K}$ is a filtering family, and thus, Theorem~\ref{thm:essential_finiteness} ensures there exists $L \in \mc{K}$ and $G_{\mf{a}} \le M \le L$ such that $M$ is $G$-invariant and open in $L$ and $M/G_{\mf{a}}$ is compact.  

We now consider the series
\[
\triv \le G_{\mf{a}} \le M \le L \le G.
\]
Theorem~\ref{thmintro:Schreier_refinement} implies that one of $L/M$, $M/G_{\mf{a}}$, or $G_{\mf{a}}/\triv$ covers $\mf{a}$.  As $\mf{a}$ is non-negligible, neither the discrete factor $L/M$ nor the compact factor $M/G_{\mf{a}}$ covers $\mf{a}$.  We deduce that $G_{\mf{a}}$ covers $\mf{a}$, hence every closed normal subgroup of $G$ that contains $G_{\mf{a}}$ covers $\mf{a}$.  Since every closed normal subgroup of $G$ that covers $\mf{a}$ contains $G_{\mf{a}}$ by construction, the proposition is verified.
\end{proof}


\bibliographystyle{amsplain}
\bibliography{biblio}{}

\end{document}